\documentclass[12pt]{amsart}
\usepackage{amssymb}
\usepackage{enumitem}
\usepackage[all,cmtip]{xy}

\makeatletter
\@namedef{subjclassname@2010}{%
  \textup{2010} Mathematics Subject Classification}
\makeatother

\newtheorem{theorem}{Theorem}[section]

\newtheorem{lemma}[theorem]{Lemma}
\newtheorem{proposition}[theorem]{Proposition}

\theoremstyle{definition}

\newtheorem{remark}[theorem]{Remark}
\newtheorem{example}[theorem]{Example}

\numberwithin{equation}{section}

\providecommand{\U}[1]{\protect\rule{.1in}{.1in}}

\begin{document}
\baselineskip=17pt

\pagestyle{myheadings}
\markboth{L.M. C\^{a}mara and M.A.S. Ruas}{On the moduli space of quasi-homogeneous functions}

\title{On the moduli space of quasi-homogeneous functions}
\author[L.M. C\^{a}mara]{Leonardo Meireles C\^{a}mara}
\address{Departamento de Matem\'{a}tica, Centro de Ci\^{e}ncias
   Exatas, Universidade Federal do Esp\'{\i}rito Santo, CEP 29075-910,
   Vit\'{o}ria-ES, Brasil}
\email{ leonardo.camara@ufes.br}
\author[M.A.S. Ruas]{Maria Aparecida Soares Ruas}
\address{Instituto de Ci\^{e}ncias
  Matem\'{a}ticas e de Computac\~{a}o,Universidade de S\~{a}o Paulo,
  Centro, Caixa Postal: 668, CEP 13560-970,S\~{a}o Carlos-SP, Brasil}
\email{maasruas@icmc.usp.br}
\begin{abstract}
We relate the moduli space of analytic equivalent germs of reduced
quasi-homogeneous functions at $(\mathbb{C}^{2},0)$ with their bi-Lipschitz
equivalence classes. We show that any non-degenerate continuous family of
(reduced) quasi-homogeneous functions with constant Henry-Parusi\'{n}ski
invariant is analytically trivial. Further we show that there are only a
finite number of distinct bi-Lipschitz classes among quasi-homogeneous
functions with the same Henry-Parusi\'{n}ski invariant providing a maximum
quota for this number.
\end{abstract}

\maketitle

\section{Preliminaries\label{prelim.}}

The main goal of this paper is to relate the moduli space of analytically
equivalent germs of quasi-homogeneous functions with their bi-Lipschitz
classes. In order to turn our statements more precise, and to give appropriate
answers, we need to introduce some terminology.

We say that two function germs $f,g\in\mathcal{O}_{2}$ are
\textit{analytically} \textit{equivalent} if there is $\Phi\in
\operatorname*{Diff}(\mathbb{C}^{2},0)$ such that $g=f\circ\Phi$. In this
case, $\Phi$ is said to be an \textit{analytic equivalence} between $f$ and
$g$. A germ of holomorphic function $f$ is said to be
\textit{quasi-homogeneous} if it is analytically equivalent to a
quasi-homogeneous polynomial. More precisely, if there is a local system of
coordinates in which $f$ can be written in the form $f(x,y)=\sum
_{ai+bj=d}a_{ij}x^{i}y^{j}$ where $a,b,d\in\mathbb{N}$.


Recall that a germ of holomorphic function $f \in \mathcal {O}_2$ is  reduced if 
 it has isolated singularity. Up to a permutation of the variables $x$ and $y$, any reduced quasi-homogeneous polynomial $f$ with weights $(p,q)$ can be (uniquely) written in the form%
\begin{equation}\label{eq_normal_form1}
f(x,y)=x^{m}y^{k}%
{\displaystyle\prod\limits_{j=1}^{n}}
(y^{p}-\lambda_{j}x^{q})
\end{equation}
where $m,k\in\{0,1\}$, $p,q\in\mathbb{Z}_{+}$, $p\leq q$, $\gcd(p,q)=1$, and
$\lambda_{j}\in\mathbb{C}^{\ast}$ are pairwise distinct. In particular
$C=f^{-1}(0)$ has $n+m+k$ distinct branches. The triple $(p,q,n)$ is clearly
an analytic invariant of the curve. In case $m=k=0$, this polynomial is called
\emph{commode} (\cite{Ar1980}, \cite{CaSc2011}). Up to a local change of
coordinates, the normal forms given by the sets \textrm{I)}, \textrm{II)} and
\textrm{III)} below determine a stratification of the moduli space of
analytically equivalent germs of singular quasi-homogeneous functions (cf.
\cite{CaSc2011}, \cite{Kang}):%

$$
\begin{array}
[c]{l}%
\text{\textrm{I) }} $k=0$\,\, \text{ and}\,\,   p=q=1,\,\, \text{i.e.},\,\, (\ref{eq_normal_form1})\,\, \text{reduces to }\,\, f(x,y)= x^{m}%
{\displaystyle\prod\limits_{j=1}^{n}}
(y-\lambda_{j}x),\\
\text{\textrm{II) }} $k=0$\,\, \text{ and}\,\,   1=p < q,\,\, \text{i.e.},\,\, (\ref{eq_normal_form1})\,\, \text{reduces to }\,\, f(x,y)=x^{m}%
{\displaystyle\prod\limits_{j=1}^{n}}
(y-\lambda_{j}x^{q}),\\
\text{\textrm{III) }}    1< p < q,\,\, \text{i.e.},\,\, (\ref{eq_normal_form1})\,\, \text{reduces to } f(x,y)=x^{m}y^{k}%
{\displaystyle\prod\limits_{j=1}^{n}}
(y^{p}-\lambda_{j}x^{q})\text{.}%
\end{array}
$$
Therefore, a reduced quasi-homogeneous function is
said to be of {type \textrm{I}, \textrm{II},} or \textrm{{III}} respectively
in cases I), II), or III) above. Note that each $\lambda_{j}$ is a root of the
polynomial $Q_{\lambda}(t)=\prod\nolimits_{j=1}^{n}(t-\lambda_{j})$, uniquely
determined by $f(x,y)$.

\begin{theorem}
[\cite{CaSc2011}]\label{CaSc}The analytic moduli space of germs of reduced
quasi-homogeneous functions is given respectively by\medskip

\begin{enumerate}
\item[(1)] $\frac{\operatorname*{Symm}(\mathbb{P}_{\Delta}^{1}(n))}%
{\operatorname*{PSL}(2,\mathbb{C})}$ for function germs of type \textrm{I}%
;\bigskip

\item[(2)] $\mathbb{Z}_{2}\times\frac{\operatorname*{Symm}(\mathbb{C}_{\Delta
}(n))}{\operatorname*{Aff}(\mathbb{C})}$ for function germs of type
\textrm{II};\bigskip

\item[(3)] $\mathbb{Z}_{2}\times\mathbb{Z}_{2}\times\frac{\operatorname*{Symm}%
(\mathbb{C}_{\Delta}^{\ast}(n))}{\operatorname*{GL}(1,\mathbb{C})}$ for
function germs of type \textrm{III},
\end{enumerate}

where for $M=\mathbb{P}^{1},\mathbb{C},\mathbb{C}^{\ast}$%
,$\,\,\operatorname{Symm}(M_{\Delta}(n))$ is the quotient space by the action
of the symmetric group $S_{n}$ on $M_{\Delta}(n):=\{(x_{1},\cdots,x_{n})\in
M^{n}:x_{i}\neq x_{j}$ for all $i\neq j\}$ given by $(\sigma,x)\mapsto
\sigma\cdot x=(x_{\sigma(1)},\cdots,x_{\sigma(n)})$.
\end{theorem}

The analytic classification of the non-reduced case can be achieved by a
similar reasoning (\cite{CaSc2011}).

Now let us turn to the bi-Lipschitz equivalence. Two function-germs
$f,g\in\mathcal{O}_{2}$ are called \emph{bi-Lipschitz equivalent} if there
exists a bi-Lipschitz map-germ $\phi\colon(\mathbb{C}^{2},0)\rightarrow
(\mathbb{C}^{2},0)$ such that $f=g\circ\phi$.

In \cite{PaHe2003}, Henry and Parusi\'{n}ski showed that the bi-Lipschitz
equivalence of analytic function germs $f:(\mathbb{C}^{2},0)\rightarrow
(\mathbb{C},0)$ admits continuous moduli. They introduce a new invariant based
on the observation that a bi-Lipschitz homeomorphism does not move much the
regions around the relative polar curves. For a single germ $f$ defined at
$(\mathbb{C}^{2},0)\,$, the invariant is given in terms of the leading
coefficients of the asymptotic expansions of $f$ along the branches of its
generic polar curve.

Fernandes and Ruas in \cite{FeRu2011} study the \emph{strong bi-Lipschitz
triviality:} two function germs $f$ and $g$ are \emph{strongly bi-Lipschitz
equivalent} if they can be embedded in a bi-Lipschitz trivial family, whose
triviality is given by integrating a Lipschitz vector field. They show that if
two quasi-homogeneous (but not homogeneous) function-germs $(\mathbb{C}%
^{2},0)\rightarrow(\mathbb{C},0)$ are strongly bi-Lipschitz equivalent then
they are analytically equivalent. This result does not hold for families of
homogeneous germs with isolated singularities and same degree since they are
Lipschitz equivalent (\cite{FeRu2004}).

In some sense, for quasi-homogeneous (but not homogeneous) real
function-germs in two variables the problem of bi-Lipschitz classification is
quite close to the problem of analytic classification.

The moduli space of bi-Lipschitz equivalence is not completely understood yet,
not even in the case of quasi-homogeneous function germs. This is the
question we address in this paper. For quasi-homogeneous (but not homogeneous) polynomials in the
plane, we ask whether Henry-Parusi\'{n}ski's invariant characterizes
completely their bi-Lipschitz class. In Theorems \ref{type II} and
\ref{type III} we compare the moduli spaces of analytic and bi-Lipschitz
equivalences. Finally, in Theorem \ref{cont. family} we prove that a
bi-Lipschitz trivial family of quasi-homogeneous (but not homogeneous)
function germs in the plane is analytically trivial.

In what follows, we assume that the quasi-homogeneous polynomial $f$ is
miniregular in $y$ in the sense that $x=0$ is not a line in the tangent cone
of the curve $(f=0)$. Finally, we note that two reduced quasi-homogeneous
function germs of the form \eqref{eq_normal_form1} are bi-Lipschitz equivalent if and only if
their corresponding commode  parts ($k=m=0$) are equivalent (\cite{FeRu2004}).

\section{Polar curves and bi-Lipschitz invariants}

Taking into account the remark in the last paragraph, it is enough to consider
germs of commode quasi-homogeneous functions, i.e., germs of functions as
in  \eqref{eq_normal_form1}, with $k=m=0$ and $1\leq p < q.$ Here we
describe some bi-Lipschitz invariants of these polynomials.

\paragraph{\textbf{The polynomial }$Q_{\lambda}(t)$.}

Let 
\begin{equation}
f_{\lambda}(x,y)={\displaystyle\prod\limits_{j=1}^{n}}(y^{p}-\lambda_{j}x^{q}),
\label{eq18a}%
\end{equation}
where $p,q\in\mathbb{Z}_{+}$, $1\leq p<q$, $\gcd(p,q)=1$, and
$\lambda_{j}\in \mathbb C^{*}.$
Let $\partial_{y}f_{\lambda}=0$ be the polar curve with respect to the
direction $(0:1)\in\mathbb{P}^{1}$.

Now let $t=y^{p}/x^{q}$, then $f_{\lambda}(x,y)=x^{nq}Q_{\lambda}(t)$, where
$Q_{\lambda}(t)={\prod_{j=1}^{n}}(t-\lambda_{j})$. The following hold:

\begin{proposition}
\label{lambdas x kappas}

\begin{enumerate}
\item[a)] $f_{\lambda}(x,y)$ determines uniquely $Q_{\lambda}(t)$.
Reciprocally, for every pair of natural numbers $1\leq p < q$, $\gcd(p,q)=1$,
$Q_{\lambda}(t)$ determines a unique polynomial $f_{\lambda}(x,y)$;

\item[b)] Let $\kappa=(\kappa_{1},\cdots,\kappa_{n-1})$ be all the
(non-necessarily distinct) critical points of $Q_{\lambda}(t)$. Then
$Q_{\lambda}^{\prime}(t)=n{\prod_{j=1}^{n-1}}(t-\kappa_{j})$ and
\[
\partial_{y}f_{\lambda}(x,y)=pny^{p-1}{\prod\limits_{j=1}^{n-1}}(y^{p}%
-\kappa_{j}x^{q});
\]

\item[c)] Moreover, with $\kappa=(\kappa_{1},\cdots,\kappa_{n-1})$ one has:
\begin{equation}
\sigma_{\ell}(\kappa)=\frac{n-\ell}{n}\sigma_{\ell}(\lambda),\,\ell
=0,\ldots,n-1, \label{eq17}%
\end{equation}
where $\sigma_{\ell}$ is the elementary symmetric polynomial of degree $\ell$.
\end{enumerate}
\end{proposition}

\begin{proof}
Condition a) follows immediately from the equation $f_{\lambda}(x,y)=x^{nq}\cdot Q_{\lambda}(y^p/x^q)$. Now $Q_{\lambda
}(t)={\prod_{j=1}^{n}}(t-\lambda_{j})$ may also be written in terms of the
elementary symmetric polynomials $\sigma_{j}(\lambda)$ as $Q_{\lambda}%
(t)=\sum_{j=0}^{n}(-1)^{j}\sigma_{j}(\lambda)t^{n-j}$, $\sigma_{0}(\lambda
)=1$. Notice that $Q_{\lambda}^{\prime}(t)$ is a degree $n-1$ polynomial in
the variable $t$, thus we can write $Q_{\lambda}^{\prime}(t)=n{\prod
_{j=1}^{n-1}}(t-\kappa_{j})$. The
conditions (b) and (c) follow immediately from this.
\end{proof}

\paragraph{\textbf{The Henry-Parusi\'{n}ski invariants.}}

We keep the notation of Proposition \ref{lambdas x kappas} and order the
$\kappa_{i}$'s so that the ones equal to zero appear on places $r+1,\ldots
,n-1$. More precisely, the irreducible branches of the polar curve
$\partial_{y}f_{\lambda}(x,y)=0$ are given by $y^{p}-\kappa_{\ell}x^{q}=0$,
for $\ell=1,\ldots,r$; if $p>1$, one more branch is also given by $y=0$. Their
Puiseux parametrisations are given by, respectively, $\gamma_{\ell}%
(s)=(s^{p},\alpha_{\ell}s^{q})$, where $\alpha_{\ell}=\sqrt[p]{\kappa_{\ell}%
}\neq0$, and by $\gamma_{0}(s)=(s,0)$.

When $f$ is as in \eqref{eq18a}, the tangent cone of $f(x,y)=0$ contains only one singular line given by $y=0$. Notice that when $f$ is a
reduced homogeneous polynomial then the tangent cone of $f(x,y)=0$ contains
no singular line.

Recall from \cite[p. 225]{PaHe2003} that the Henry-Parusi\'{n}ski invariants
are obtained in the following way. For each polar arc $\gamma$ tangent to a
singular line of the tangent cone of $f(x,y)=0$, we associate two numbers:
$h=h(\gamma)\in\mathbb{Q}_{+}$ and $c=c(\gamma)\in\mathbb{C}^{\ast}$ given by
the expansions $f(\gamma(s))=cs^{h}+\ldots,\,c\neq0$. In particular, the
Henry-Parusi\'{n}ski invariants are not defined for reduced homogeneous germs.

Since $f$ is quasi-homogeneous, $h$ is determined by the weights, so we can
omit it in the definition of the invariants. The invariant of
Henry-Parusi\'{n}ski of a quasi-homogeneous $f$ is the set $\text{Inv}%
(f)=\frac{\{c_{0},c_{1},\ldots,c_{r}\}}{\mathbb{C}^{\ast}}$ for germs of type
$\mathrm{III}$ (i.e., $p>1$ in \eqref{eq18a})  and $\text{Inv}(f)=\frac{\{c_{1},\ldots,c_{r}%
\}}{\mathbb{C}^{\ast}}$ for germs of type $\mathrm{II}$, (i.e. $p=1$ in \eqref{eq18a}), where
$c_{j}$ is the coefficient of the leading term of $f(\gamma_{j}(s)).$ If
$p>1$, two sets $\{c_{0},c_{1},\ldots,c_{r}\}$ and $\{c_{1}^{\prime}%
,\ldots,c_{r}^{\prime},c_{0}^{\prime}\}$ define the same Henry-Parusi\'{n}ski
invariant if there exists $\xi\in\mathbb{C}^{\ast}$ such that
\begin{equation}
\{c_{0}^{\prime},c_{1}^{\prime},\ldots,c_{r}^{\prime}\}=\{c_{0}\xi^{qn}%
,c_{1}\xi^{pqn}\ldots,c_{r}\xi^{pqn}\}. \label{eq19}%
\end{equation}
If $p=1$, two sets $\{c_{1},c_{2},\ldots,c_{r}\}$ and $\{c_{1}^{\prime}%
,c_{2}^{\prime},\ldots,c_{r}^{\prime}\}$ define the same Henry-Parusi\'{n}ski
invariant if there exists $\xi\in\mathbb{C}^{\ast}$ such that
\begin{equation}
\{c_{1}^{\prime},c_{2}^{\prime},\ldots,c_{r}^{\prime}\}=\{c_{1}\xi^{qn}%
,c_{2}\xi^{qn},\ldots,c_{r}\xi^{qn}\}. \label{eq19bis}%
\end{equation}
We have $f\circ\gamma
_{0}(s)=(-1)^{n}\lambda_{1}\cdots\lambda_{n}s^{nq}$ (i.e., if $p>1$) and
$f\circ\gamma_{\ell}(s)=\left[  \prod\nolimits_{j=1}^{n}(\kappa_{\ell}%
-\lambda_{j})\right]  s^{pqn}$. Thus we define
\begin{equation}
\rho_{0}:=(-1)^{n}\lambda_{1}\cdots\lambda_{n}\text{ and }\rho_{\ell}%
:=\prod\nolimits_{j=1}^{n}(\kappa_{\ell}-\lambda_{j}). \label{eq18}%
\end{equation}

\begin{remark}
Although the polar arcs are not necessarily reduced, as we shall see in the
next paragraph, this is precisely what happens in the generic case.
\end{remark}

\paragraph{\textbf{The generic polynomial in two variables.}}

Let $\operatorname*{H}_{p,q}^{d}$ be the set of commode monic and reduced
quasi-homogeneous polynomials in two variables with relatively prime weights
$(p,q)$, $1\leq p<q$, and total degree $d$. Then $\operatorname*{H}_{p,q}^{d}$
is an affine space of dimension $n$. From Proposition \ref{lambdas x kappas},
there is an isomorphism $\operatorname*{H}\nolimits_{p,q}^{d}\simeq
\operatorname*{P}\nolimits_{1}^{n}$, where $\operatorname*{P}\nolimits_{1}%
^{n}$ is the affine space of monic polynomials of degree $n$, say $Q_{\sigma
}(t)=t^{n}+\sum_{j=1}^{n}(-1)^{j}\sigma_{j}t^{n-j}$. We write $\psi
:\operatorname*{H}\nolimits_{p,q}^{d}\overset{\simeq}{\longrightarrow
}\operatorname*{P}\nolimits_{1}^{n}$ and $\pi:\operatorname*{P}\nolimits_{1}%
^{n}\longrightarrow\mathbb{C}^{n}$, $\psi(f)=Q$ and $\pi(Q)=\sigma
:=(\sigma_{1},\cdots,\sigma_{n})$.

The following result holds true:

\begin{theorem}
There is a Zariski open set $\mathcal{Z}\subset\operatorname*{H}_{p,q}^{d}$
such that for every $f\in\mathcal{Z}$ the following holds:

\begin{enumerate}
\item[a)] the polar curve $\partial_{y}f_{x}=0$ has $n-1$ distinct roots;

\item[b)] For $\ell\neq\ell^{\prime}$, $\rho_{\ell}\neq\rho_{\ell^{\prime}},$
$1\leq\ell\leq n-1.$
\end{enumerate}
\end{theorem}

\begin{proof}
The family $Q_{\sigma}:=Q_{\sigma}(t)$ is the versal unfolding of
$Q_{0}(t)=t^{n}$. Therefore, as usual, we can define
\begin{enumerate}
\item $B_{L}:=\{\sigma\in\mathbb{C}^{n}:\exists t_{0}\in\mathbb{C}$ such that
$Q_{\sigma}^{\prime}(t_{0})=Q_{\sigma}^{\prime\prime}(t_{0})=0\}$;
\item $B_{G}:=\{\sigma\in\mathbb{C}^{n}:\exists t_{1},t_{2}\in\mathbb{C}$ such
that $Q_{\sigma}^{\prime}(t_{1})=Q_{\sigma}^{\prime}(t_{2})=0$ and $Q_{\sigma
}(t_{1})=Q_{\sigma}(t_{2})\}$;
\end{enumerate}
the local and semi-local subsets of the total bifurcation set $B=B_{L}\cup
B_{G}$ of $Q_{0}$. Since they are proper algebraic sets of $\mathbb{C}^{n}$,
we can take $Z_{0}=Z_{L}\cap Z_{G}$, where $Z_{L}$ is the Zariski open set
given by the complement $(\pi\circ\psi)^{-1}(B_{L})$ and $Z_{G}$ is the
complement $(\pi\circ\psi)^{-1}(B_{G})$. Polynomials in $Z_{L}$ satisfy a) and
those in $Z_{G}$ satisfy b). Then $Z_{0}$ satisfies the required conditions.
\end{proof}

\begin{remark}
Notice that $B\subset\mathbb{C}^{n}$ has a stratification which induces a
stratification in $(\pi\circ\psi)^{-1}(B)$.
\end{remark}

\section{Analytic moduli}

The analytic classification of quasi-homogeneous function-germs at the origin of $\mathbb{C}^{2}$
was given in \cite{Kang} and \cite{CaSc2011} (cf. Theorem \ref{CaSc}). We
revisit the analytic classification from \cite{CaSc2011} under the light of
bi-Lipschitz invariants.

Let $M$ be a manifold and $M_{\Delta}(n):=\{(x_{1},\cdots,x_{n})\in
M^{n}:x_{i}\neq x_{j}$ for all $i\neq j\}$. Let $S_{n}$ denote the group of
$n$ elements and consider its action on $M_{\Delta}(n)$ given by
$(\sigma,\lambda)\mapsto\sigma\cdot\lambda=(\lambda_{\sigma(1)},\cdots
,\lambda_{\sigma(n)})$. The quotient space induced by this action is denoted
by $\operatorname*{Symm}(M_{\Delta}(n))$. Now suppose a Lie group $G$ acts on
$M$ and let $G$ act on $M_{\Delta}(n)$ in the natural way $(g,\lambda
)=(g\cdot\lambda_{1},\cdots,g\cdot\lambda_{n})$ for every $\lambda\in
M_{\Delta}(n)$. Then the actions of $G$ and $S_{n}$ on $M_{\Delta}(n)$
commute. Thus we obtain a natural action of $G$ on $\operatorname*{Symm}%
(M_{\Delta}(n))$. Given $\lambda\in M_{\Delta}(n)$, denote its equivalence
class in $\operatorname*{Symm}(M_{\Delta}(n)))/G$ by $[\lambda]$.

Here we present some useful distinct characterizations of the analytic moduli
space of reduced quasi-homogeneous function-germs.

\paragraph{\textbf{The moduli space }$\mathcal{M}_{[n],1}^{0}$.}

Let $\mathcal{M}_{[n],1}^{0}$ denote the moduli space of analytically
equivalent punctured Riemann spheres with one marked puncture and $n$
unordered punctures, i.e., $\mathcal{M}_{[n],1}^{0}=\frac{\operatorname*{Symm}%
(\mathbb{C}_{\Delta}(n))}{\operatorname*{Aff}(\mathbb{C})}$. From Theorem
\ref{CaSc} this coincides with the moduli space of quasi-homogeneous functions
of type II. Now we give a suitable description of the space $\mathcal{M}%
_{[n],1}^{0}$.

Let $\mathbb{W}_{0}^{n-1}:=\{\lambda=(\lambda_{1},\cdots,\lambda_{n}%
)\in\operatorname*{Symm}(\mathbb{C}_{\Delta}(n)):\lambda_{1}+\cdots
+\lambda_{n}=0\}$ and consider the natural action of $\mathbb{C}^{\ast}$ on
$\mathbb{C}^{n}$ given by $(\lambda,z)\mapsto\lambda z=(\lambda z_{1}%
,\cdots,\lambda z_{n})$. Clearly, this action induces an analogous action on
$\mathbb{W}_{0}^{n-1}$. Their cosets induce the following isomorphism.

\begin{lemma}
The moduli space $\mathcal{M}_{[n],1}^{0}$ is isomorphic to $\mathbb{W}_{0}^{n-1}%
\slash \operatorname*{GL}(1,\mathbb{C})$.


\end{lemma}

\begin{proof}
First notice that each class $[\lambda]\in\frac{\operatorname*{Symm}%
(\mathbb{C}_{\Delta}(n))}{\operatorname*{Aff}(\mathbb{C})}$ has a
representative $[\lambda_{0}]\in\mathbb{W}_{0}^{n-1}$, thus it is enough to
show that $[\lambda]=[\lambda^{\prime}]\in\frac{\operatorname*{Symm}%
(\mathbb{C}_{\Delta}(n))}{\operatorname*{Aff}(\mathbb{C})}$ if and only if
$[\lambda_{0}]=[\lambda_{0}^{\prime}]\in\frac{\mathbb{W}_{0}^{n-1}%
}{\operatorname*{GL}(1,\mathbb{C})}$. Let $\Phi:\operatorname*{Symm}%
(\mathbb{C}_{\Delta}(n))\longrightarrow\mathbb{W}_{0}^{n-1}$ be the map given
by $\lambda=(\lambda_{1},\cdots,\lambda_{n})\mapsto\lambda_{0}:=\lambda
-\frac{\lambda_{1}+\cdots+\lambda_{n}}{n}(1,\cdots,1)$. Then a straightforward
calculation shows that $\Phi(a\lambda+b)=a\Phi(\lambda)$, i.e., if
$\lambda^{\prime}=a\lambda+b$, then $\lambda_{0}^{\prime}=a\lambda_{0}$. Thus
$\Phi$ induces a well defined map $\overline{\Phi}:\frac{\operatorname*{Symm}%
(\mathbb{C}_{\Delta}(n))}{\operatorname*{Aff}(\mathbb{C})}\longrightarrow
\frac{\mathbb{W}_{0}^{n-1}}{\operatorname*{GL}(1,\mathbb{C})}$. This map is
clearly surjective, thus it suffices to show that it is injective. Now let
$\lambda,\lambda^{\prime}\in\operatorname*{Symm}(\mathbb{C}_{\Delta}(n))$ be
such that $\Phi(\lambda^{\prime})=a\Phi(\lambda)$, this implies the existence
of $b,b^{\prime}\in\mathbb{C}$ such that $\lambda^{\prime}-b^{\prime}%
(1,\cdots,1)=a(\lambda-b(1,\cdots,1))$. The result then follows.
\end{proof}

Now consider the above defined actions of $\mathbb{Z}_{n}$ on $\mathbb{C}^{m}$
and $\operatorname*{Symm}(\mathbb{C}^{m})$, and let $\mathbb{W}_{0,1}%
^{n-2}\subset\mathbb{W}_{0}^{n-1}$ be given by $\mathbb{W}_{0,1}%
^{n-2}=\{\lambda=(\lambda_{1},\cdots,\lambda_{n})\in\mathbb{W}_{0}%
^{n-1}:\lambda_{1}\cdots\lambda_{n}=1\}$, then we can reduce $\mathcal{M}^{0}_{[n],1}$ a bit further.

\begin{lemma}
The moduli space $\mathcal{M}_{n,[1]}^{0}$ is isomorphic to $\mathbb{W}_{0,1}^{n-2}\slash\mathbb{Z}_{n}$.

\end{lemma}

\begin{proof}
Consider the map $\Psi:\mathbb{W}_{0}^{n}\longrightarrow\mathbb{W}_{0,1}%
^{n-2}$ given by $\lambda=(\lambda_{1},\cdots,\lambda_{n})\mapsto
\overline{\lambda}=\frac{1}{\sqrt[n]{\lambda_{1}\cdots\lambda_{n}}}\lambda$.
Similarly, it is enough to show that $[\lambda]=[\lambda^{\prime}]\in
\frac{\mathbb{W}_{0}^{n-1}}{\operatorname*{GL}(1,\mathbb{C})}$ if and only if
$\left[  \overline{\lambda}\right]  =\left[  \overline{\lambda^{\prime}%
}\right]  \in\frac{\mathbb{W}_{0,1}^{n-2}}{\mathbb{Z}_{n}}$. If $\lambda
^{\prime}=a\lambda$, then
\[
\overline{\lambda^{\prime}}=\frac{1}{\sqrt[n]{\lambda_{1}^{\prime}%
\cdots\lambda_{n}^{\prime}}}\lambda^{\prime}=\frac{a}{\omega^{-m}\cdot a}%
\frac{1}{\sqrt[n]{\lambda_{1}\cdots\lambda_{n}}}\lambda=\omega^{m}%
\cdot\overline{\lambda},
\]
where $\omega=\exp(2\pi\boldsymbol{i}/n)$. Thus $\Psi$ induces a map
$\overline{\Psi}:\frac{\mathbb{W}_{0}^{n-1}}{\operatorname*{GL}(1,\mathbb{C}%
)}\longrightarrow\frac{\mathbb{W}_{0,1}^{n-2}}{\mathbb{Z}_{n}}$. Clearly,
$\overline{\Psi}$ is surjective, thus it suffices to prove that it is
injective. Now notice that $\overline{\lambda^{\prime}}=\omega^{k}%
\overline{\lambda}$ implies
\[
\frac{\lambda^{\prime}}{\sqrt[n]{\lambda_{1}^{\prime}\cdots\lambda_{n}%
^{\prime}}}=\frac{\omega^{k}\lambda}{\sqrt[n]{\lambda_{1}\cdots\lambda_{n}}%
}\text{.}%
\]
But this is equivalent to say that $\lambda^{\prime}=a\lambda$ for some
$a\in\mathbb{C}$.
\end{proof}

Now recall the natural action $\mathbb{Z}_{n}$ on $\mathbb{C}^{m}$ given by
$\varsigma^{s}\mapsto(\varsigma^{s}\cdot\mu_{1},\varsigma^{s}\cdot\mu
_{2},\cdots\varsigma^{s}\cdot\mu_{m})$ with $\varsigma=\exp(\frac
{2\pi\boldsymbol{i}}{n})$, whose orbit space is denoted by $\frac
{\mathbb{C}^{m}}{\mathbb{Z}_{n}}$, then the previous lemma leads to the following isomorphism.

\begin{lemma}
\label{Kang} The map $\Xi:\mathbb{W}_{0,1}^{n-2}\longrightarrow\mathbb{W}%
_{0}^{n-2}$ given by $\lambda=(\lambda_{1},\cdots,\lambda_{n})\mapsto
\kappa=(\kappa_{1},\cdots,\kappa_{n})$, where $\sigma_{\ell}(\kappa
)=\frac{n-\ell}{n}\sigma_{\ell}(\lambda)$ for all $\ell=1,\ldots,n-1$, induces
the following isomorphism:%
\begin{equation}
\mathcal{M}_{[n],1}^{0}\simeq\frac{\mathbb{W}_{0}^{n-2}}{\mathbb{Z}_{n}}.
\label{eq29}%
\end{equation}

\end{lemma}

\begin{proof}
Since $\frac{n}{n-\ell
}\sigma_{\ell}(\kappa)$, $\ell=1,\ldots,n-1$, are the coefficients of the
monic polynomial $z^{n}+\frac{n}{n-1}\sigma_{1}(\kappa)z^{n-1}+\cdots+\frac
{n}{n-(n-1)}\sigma_{n-1}(\kappa)z+1$ having $\{\lambda_{1},\cdots,\lambda
_{n}\}$ as roots, then $\Xi$ induces a bijective map $\overline{\Xi}%
:\frac{\mathbb{W}_{0,1}^{n-2}}{\mathbb{Z}_{n}}\longrightarrow\frac
{\mathbb{W}_{0}^{n-2}}{\mathbb{Z}_{n}}$.
\end{proof}

\paragraph{\textbf{The moduli space }$\mathcal{M}_{[n],2}^{0}$}

Let $\mathcal{M}_{[n],2}^{0}$ denote the moduli space of analytically
equivalent punctured Riemann spheres with two ordered punctures and $n$
unordered punctures, i.e., $\mathcal{M}_{[n],2}^{0}=\frac{\operatorname*{Symm}%
(\mathbb{C}_{\Delta}^{\ast}(n))}{\operatorname*{GL}(1,\mathbb{C})}$. From
Theorem \ref{CaSc}, this coincides with the moduli space of quasi-homogeneous
functions of type III. Let us give a suitable description of $\mathcal{M}%
_{[n],2}^{0}$.

On the analytic variety $\mathbb{V}^{n-1}=\{\lambda\in\mathbb{C}_{\Delta
}^{\ast}(n):\lambda_{1}\cdots\lambda_{n}=1\}$, we may consider the (effective)
action of $\mathbb{Z}_{n}$ given by the multiplication by the $n$-th roots of
unity, i.e., $\varsigma^{s}\mapsto\varsigma^{s}\cdot\lambda$ with
$\varsigma=\exp(\frac{2\pi\boldsymbol{i}}{n})$. Therefore, we have the
following isomorphism.

\begin{lemma}
\label{Kang2}Let $\mathbb{V}^{n-1}=\{\lambda\in\operatorname*{Symm}%
\mathbb{C}_{\Delta}^{\ast}(n):\lambda_{1}\cdots\lambda_{n}=1\}$, then%
\begin{equation}
\mathcal{M}_{[n],2}^{0}\simeq\frac{\mathbb{V}^{n-1}}{\mathbb{Z}_{n}}.
\label{eq23}%
\end{equation}

\end{lemma}

\begin{proof}
The map $\Psi:\operatorname*{Symm}\mathbb{C}_{\Delta}^{\ast}(n)\longrightarrow
\mathbb{V}^{n-1}$ given by $\lambda=(\lambda_{1},\cdots,\lambda_{n}%
)\mapsto\overline{\lambda}=\frac{1}{\sqrt[n]{\lambda_{1}\cdots\lambda_{n}}%
}\lambda$ induces a well defined map $\widetilde{\Psi}:\frac
{\operatorname*{Symm}(\mathbb{C}_{\Delta}^{\ast}(n))}{\operatorname*{GL}%
(1,\mathbb{C})}\longrightarrow\frac{\mathbb{V}^{n-1}}{\mathbb{Z}_{n}}$. In
fact, $[\lambda^{\prime}],[\lambda]\in\mathbb{V}^{n-1}$ are equivalent in
$\operatorname*{Symm}(\mathbb{C}_{\Delta}^{\ast}(n))$ if and only if there is
$\alpha\in\mathbb{C}^{\ast}$ such that $\lambda^{\prime}=\alpha\lambda$. Since
$\lambda_{1}\cdots\lambda_{n}=1=\lambda_{1}^{\prime}\cdots\lambda_{n}^{\prime
}$, then $\alpha^{n}=1$. As before, it is enough to show that $[\lambda
]=[\lambda^{\prime}]\in\frac{\operatorname*{Symm}\mathbb{C}_{\Delta}^{\ast
}(n)}{\operatorname*{GL}(1,\mathbb{C})}$ if and only if $\left[
\overline{\lambda}\right]  =\left[  \overline{\lambda^{\prime}}\right]
\in\frac{\mathbb{V}^{n-1}}{\mathbb{Z}_{n}}$. Since $\Psi$ is clearly
surjective, we only have to prove that it is injective. In fact, suppose
$\left[  \overline{\lambda}\right]  =\left[  \overline{\lambda^{\prime}%
}\right]  \in\frac{\mathbb{V}^{n-1}}{\mathbb{Z}_{n}}$, then
\[
\frac{1}{\sqrt[n]{\lambda_{1}^{\prime}\cdots\lambda_{n}^{\prime}}}%
\lambda^{\prime}=\overline{\lambda^{\prime}}=\alpha\overline{\lambda}\frac
{1}{\sqrt[n]{\lambda_{1}\cdots\lambda_{n}}}\lambda,\quad\alpha^{n}=1.
\]
Therefore, $\lambda^{\prime}=a\lambda$ for some $a\in\mathbb{C}.$
\end{proof}

\section{Analytic moduli and bi-Lipschitz invariants}

In this section, we establish the relationship between the analytic moduli
space and the Henry-Parusi\'{n}ski invariant.

\subsection{Type II commode functions}

From (\ref{eq19}) and (\ref{eq18}), we obtain the correspondence%
\[
\kappa=(\kappa_{1},\cdots,\kappa_{n-1})\rightarrow\rho=(\rho_{1},\cdots
,\rho_{n-1}),
\]
where $\rho_{j}$ is given by $\rho_{\ell}=\prod\nolimits_{j=1}^{n}%
(\kappa_{\ell}-\lambda_{j})$. This correspondence is represented by the system
of equations%

\begin{equation}
\left\{
\begin{array}
[c]{l}%
(\kappa_{1})^{n}+\sum_{\ell=1}^{n-1}(-1)^{\ell}\frac{n}{n-\ell}\sigma_{\ell
}(\kappa)(\kappa_{1})^{n-\ell}+\left(  -1\right)  ^{n}\mu_{n}=\rho_{1},\\
\multicolumn{1}{c}{\vdots}\\
(\kappa_{n-1})^{n}+\sum_{\ell=1}^{n-1}(-1)^{\ell}\frac{n}{n-\ell}\sigma_{\ell
}(\kappa)(\kappa_{n-1})^{n-\ell}+\left(  -1\right)  ^{n}\mu_{n}=\rho_{n-1},
\end{array}
\right.  \label{eq3}%
\end{equation}
where $\mu_{n}=1$, and defines a map $\Upsilon:\mathbb{C}^{n-1}\longrightarrow
\mathbb{C}^{n-1}$, $\kappa_{j}\mapsto\rho_{j}$, $j=1,\ldots,n-1$.  It
follows from Lemma \ref{Kang} that in order to compare the analytic invariants
and the Henri-Parusi\'{n}ski invariants it suffices to consider the
restriction of $\Upsilon$ to the hyperplane $\mathbb{W}_{0}^{n-2}%
=\{(\kappa_{1},\cdots,\kappa_{n-1})\in\mathbb{C}^{n-1}:\kappa_{1}%
+\cdots+\kappa_{n-1}=0\}$. Notice that $\mathbb{W}_{0}^{n-2}$ is a manifold
admitting a system of coordinates given by $\overline{\kappa}=$ $(\kappa
_{1},\cdots,\kappa_{n-2})$. In particular, the last equation of (\ref{eq3})
determines immediately $\rho_{n-1}$ in terms of these $n-2$ parameters.
Therefore, $\operatorname{Im}(\Upsilon)$ may be considered as (an algebraic)
graph over $\mathbb{C}^{n-2}$ with coordinates given by $\overline{\rho}%
=(\rho_{1},\cdots,\rho_{n-2})$.

Denote the restriction of $\Upsilon$ to $\mathbb{W}_{0}^{n-2}$ by
$\Upsilon_{II}:=\left.  \Upsilon\right\vert _{\mathbb{W}_{0}^{n-2}}$, then in
the above systems of coordinates $\Upsilon_{II}=(\Upsilon_{II,1}%
,\cdots,\Upsilon_{II,\kappa-2})$ is given by%

\begin{equation}
\left\{
\begin{array}
[c]{l}%
\Upsilon_{II,1}(\kappa)=(\kappa_{1})^{n}+\sum_{\ell=1}^{n-1}(-1)^{\ell}%
\frac{n}{n-\ell}\sigma_{\ell}(\kappa)(\kappa_{1})^{n-\ell}+\left(  -1\right)
^{n},\\
\multicolumn{1}{c}{\vdots}\\
\Upsilon_{II,\kappa-2}(\kappa)=(\kappa_{n-2})^{n}+\sum_{\ell=1}^{n-1}%
(-1)^{\ell}\frac{n}{n-\ell}\sigma_{\ell}(\kappa)(\kappa_{n-2})^{n-\ell
}+\left(  -1\right)  ^{n}.
\end{array}
\right.  \label{eq3bis}%
\end{equation}

Let $\mathcal{HP}_{II}:=\operatorname{Im}(\left.  \Upsilon\right\vert
_{\mathbb{W}_{0}^{n-2}})$, then to each $\rho\in\mathcal{HP}_{II}$ there
corresponds a unique class of the Henry-Parusi\'{n}ski invariant of a type II
function germ. Besides, since $\mathbb{W}_{0}^{n-2}$ is an affine space of
dimension $n-2$, each generic point in the image of $\Upsilon$ admits
$n^{n-2}$ points in its pre-image (Bezout's theorem). From (\ref{eq3}) and
Lemma \ref{Kang}, $\Upsilon$ induces a map between the corresponding moduli
spaces, say $\overline{\Upsilon}_{II}:\mathcal{M}_{n,[1]}^{0}\longrightarrow
\mathcal{HP}_{II}$. In this case, each point $p\in\mathcal{HP}_{II}$ admits
precisely $n^{n-3}$ points in its pre-image (counting multiplicities) with
respect to $\overline{\Upsilon}_{II}$.

For any $\rho=(\rho_{1},\cdots,\rho_{n-2})\in\mathbb{C}^{n-2}$, we say that
$\Upsilon_{II}^{-1}(\rho)$ is a \textit{degenerate fiber} of $\Upsilon_{II}$
if ${\rho}_{j}=0$ for some $j\in\{1,\cdots,n-2\}$ or else if the fiber has a
multiple root; otherwise, we say that $\Upsilon_{II}^{-1}(\rho)$ is a
\textit{non-degenerate fiber}. We have

\begin{theorem}
\label{type II}For type II functions in $\operatorname*{H}_{pq}^{d}$, the
analytic moduli space with fixed Henry-Parusi\'{n}ski invariant is determined
by the equivalence classes in $\mathbb{W}_{0}^{n-2}$ of the non-degenerate
fibers of $\Upsilon_{II}:\mathbb{W}_{0}^{n-2}\rightarrow\mathcal{HP}_{II}$.
More precisely, for each $\rho\in\operatorname{Im}(\Upsilon_{II})$ there exist
$\#[\Upsilon_{II}^{-1}(\rho)]\leq n^{n-3}$ analytic types of function-germs with
the same Henry-Parusi\'{n}ski invariant $\rho$. The equality holds for generic
polynomials in $\mathcal{Z}\subset\operatorname*{H}_{pq}^{d}$.
\end{theorem}

\begin{proof}
From the description of the space $\mathcal{M}_{[n],1}^{0}$ (cf. (\ref{eq29}))
and from (\ref{eq19bis}) and (\ref{eq18}), the map $\Upsilon_{II}%
:\mathbb{W}_{0}^{n-2}\longrightarrow\mathbb{C}^{n-2}$ induces a surjective map
$\overline{\Upsilon}_{II}:\mathcal{M}_{[n],1}^{0}\longrightarrow
\mathcal{HP}_{II}$. In other words, the correspondence between the
Henry-Parusi\'{n}ski invariant and the analytic invariants is determined by
the orbits of the action of $\mathbb{Z}_{n}$ on the fibers over
$\operatorname{Im}\Upsilon_{II}$ of the homogeneous map $\Upsilon
_{II}:\mathbb{W}_{0}^{n-2}\longrightarrow\mathbb{C}^{n-2}$, given by
$\Upsilon_{II}(\kappa)=(\Upsilon_{II,1}(\kappa),\cdots,\Upsilon_{II,n-2}%
(\kappa))$, where%
\begin{equation}
\Upsilon_{II,j}(\kappa):=(\kappa_{j})^{n}+\sum_{\ell=1}^{n-1}(-1)^{\ell}%
\frac{n}{n-\ell}\sigma_{\ell}(\kappa)(\kappa_{j})^{n-\ell}+(-1)^{n}%
,\label{eq26bis}%
\end{equation}
$j=1,\ldots,n-2$. Finally, Bezout's theorem and the action of $\mathbb{Z}_{n}$
on the fibers over $\operatorname{Im}\Upsilon_{II}$ lead to the desired result
\end{proof}

Let us see some examples. Since for $n=2$ there is trivially just one analytic
class (due to classical complex analysis arguments), then we shall only
consider $n\geq3$.

\begin{example}
Suppose $n=3$, then $\kappa=(\kappa_{1},\kappa_{2})$ and $\mu_{j}=\sigma
_{j}(\lambda)$. Since $\kappa\in\mathbb{W}_{0}^{n-2}$, then $\Upsilon_{II}$ is
given by%
\[
\left\{
\begin{array}
[c]{c}%
(\kappa_{1})^{3}+(-1)^{2}\frac{3}{3-2}\sigma_{2}(\kappa)\kappa_{1}=\rho
_{1}-(-1)^{3},\\
(\kappa_{2})^{3}+(-1)^{2}\frac{3}{3-2}\sigma_{2}(\kappa)\kappa_{2}=\rho
_{2}-(-1)^{3}.%
\end{array}
\right.
\]
Since $\kappa_{2}=-\kappa_{1}$, then we have%
\[
\left\{
\begin{array}
[c]{c}%
-2(\kappa_{1})^{3}=(\kappa_{1})^{3}+3(\kappa_{1}\kappa_{2})\kappa_{1}=\rho
_{1}+1\\
-2(\kappa_{2})^{3}=(\kappa_{2})^{3}+3(\kappa_{1}\kappa_{2})\kappa_{2}=\rho
_{2}+1
\end{array}
\right.  \Leftrightarrow\left\{
\begin{array}
[c]{c}%
(\kappa_{1})^{3}=-\frac{1+\rho_{1}}{2},\\
(\kappa_{2})^{3}=-\frac{\rho_{2}+1}{2}.
\end{array}
\right.
\]
In particular, $\rho_{2}=-\rho_{1}-2$. If we let $\alpha=\sqrt[3]{\frac
{1+\rho_{1}}{2}}$ be one of the cubic roots of $\frac{1+\rho_{1}}{2}$ and
$\omega=\exp(\frac{2\pi\boldsymbol{i}}{3})$, then%
\[
\left\{
\begin{array}
[c]{c}%
\kappa_{1}=-\omega^{s}\alpha,\\
\kappa_{2}=\omega^{s}\alpha.
\end{array}
\right.  s=0,1,2
\]
Therefore, there is just $3^{3-2}=1$ analytic class corresponding to the same
Henry-Parusi\'{n}ski invariant $\rho=(\rho_{1},-2-\rho_{1})$, where $\rho
_{1}\in\operatorname{Im}\Upsilon_{II}$.
\end{example}

\begin{example}
For $n=4$, the system (\ref{eq3}) assumes the form%
\[
\left\{
\begin{array}
[c]{l}%
(\kappa_{1})^{4}+\frac{4}{4-2}\sigma_{2}(\kappa)(\kappa_{1})^{4-2}-\frac
{4}{4-3}\sigma_{3}(\kappa)(\kappa_{1})^{4-3}=\rho_{1}-\left(  -1\right)
^{4},\\
(\kappa_{2})^{4}+\frac{4}{4-2}\sigma_{2}(\kappa)(\kappa_{2})^{4-2}-\frac
{4}{4-3}\sigma_{3}(\kappa)(\kappa_{2})^{4-3}=\rho_{2}-\left(  -1\right)
^{4},\\
(\kappa_{1})^{4}+\frac{4}{4-2}\sigma_{2}(\kappa)(\kappa_{1})^{4-2}-\frac
{4}{4-3}\sigma_{3}(\kappa)(\kappa_{1})^{4-3}=\rho_{3}-\left(  -1\right)  ^{4}.
\end{array}
\right.
\]
or equivalently%
\[
\left\{
\begin{array}
[c]{l}%
(\kappa_{1})^{4}+2(\kappa_{1}\kappa_{2}+\kappa_{1}\kappa_{3}+\kappa_{2}%
\kappa_{3})(\kappa_{1})^{2}-4\kappa_{1}\kappa_{2}\kappa_{3}(\kappa_{1}%
)=\rho_{1}-1,\\
(\kappa_{2})^{4}+2(\kappa_{1}\kappa_{2}+\kappa_{1}\kappa_{3}+\kappa_{2}%
\kappa_{3})(\kappa_{2})^{2}-4\kappa_{1}\kappa_{2}\kappa_{3}(\kappa_{2}%
)=\rho_{2}-1,\\
(\kappa_{3})^{4}+2(\kappa_{1}\kappa_{2}+\kappa_{1}\kappa_{3}+\kappa_{2}%
\kappa_{3})(\kappa_{3})^{2}-4\kappa_{1}\kappa_{2}\kappa_{3}(\kappa_{3}%
)=\rho_{3}-1.
\end{array}
\right.
\]
In other words%
\[
\left\{
\begin{array}
[c]{l}%
(\kappa_{1})^{2}[(\kappa_{1})^{2}+2(\kappa_{1}\kappa_{2}+\kappa_{1}\kappa
_{3}-\kappa_{2}\kappa_{3})]=\rho_{1}-1,\\
(\kappa_{2})^{2}[(\kappa_{2})^{2}+2(\kappa_{1}\kappa_{2}-\kappa_{1}\kappa
_{3}+\kappa_{2}\kappa_{3})]=\rho_{2}-1,\\
(\kappa_{3})^{2}[(\kappa_{3})^{2}+2(-\kappa_{1}\kappa_{2}+\kappa_{1}\kappa
_{3}+\kappa_{2}\kappa_{3})]=\rho_{3}-1.
\end{array}
\right.
\]
Assuming that $\kappa_{3}=-(\kappa_{1}+\kappa_{2})$, then we have%
\[
\left\{
\begin{array}
[c]{l}%
(\kappa_{1})^{2}[2(\kappa_{2})^{2}+2\kappa_{1}\kappa_{2}-(\kappa_{1}%
)^{2}]=\rho_{1}-1,\\
(\kappa_{2})^{2}[2(\kappa_{1})^{2}+2\kappa_{1}\kappa_{2}-(\kappa_{2}%
)^{2}]=\rho_{2}-1,\\
(\kappa_{1}+\kappa_{2})^{2}[3(\kappa_{1})^{2}+4\kappa_{1}\kappa_{2}%
+3\kappa_{2}{}^{2}]=\rho_{3}-1.
\end{array}
\right.
\]
For fixed $(\rho_{1},\rho_{2})\in\operatorname{Im}(\Upsilon_{II})$ the
classical Bezout's theorem says that there are precisely $4^{2}=16$ solutions
for the above system. Further, the last one tells us which value $\rho_{3}$
must have in order that $\rho=(\rho_{1},\rho_{2},\rho_{3})$ be the associated
Henri-Parusi\'{n}ski invariant. Considering the symmetries, this information
tells us that there are $4^{4-3}=\frac{16}{4}$ distinct analytic types of
quasi-homogeneous functions with the Henry-Parusi\'{n}ski invariant $\rho$.
\end{example}

\subsection{Type III commode functions}

From (\ref{eq19}) and (\ref{eq18}), we obtain the correspondence%
\[
\kappa=(\kappa_{1},\cdots,\kappa_{n-1})\rightarrow\rho=(\rho_{1},\cdots
,\rho_{n-1}),
\]
where $\rho_{0}=(-1)^{n}\lambda_{1}\cdots\lambda_{n}$ and $\rho_{\ell}%
=\prod\nolimits_{j=1}^{n}(\kappa_{\ell}-\lambda_{j})$ for $\ell>0$. This
correspondence is represented by the map $\Upsilon_{III}:\mathbb{C}%
^{n-1}\rightarrow\mathbb{C}^{n-1}$ , $\kappa\mapsto c=(\frac{\rho_{1}}%
{\rho_{0}},\cdots,\frac{\rho_{n-1}}{\rho_{0}})$, given by%
\[
\left\{
\begin{array}
[c]{l}%
(\kappa_{1})^{n}+\sum_{\ell=1}^{n-1}(-1)^{\ell}\frac{n}{n-\ell}\sigma_{\ell
}(\kappa)(\kappa_{1})^{n-\ell}+(-1)^{n}=(-1)^{n}c_{1},\\
\multicolumn{1}{c}{\vdots}\\
(\kappa_{n-1})^{n}+\sum_{\ell=1}^{n-1}(-1)^{\ell}\frac{n}{n-\ell}\sigma_{\ell
}(\kappa)(\kappa_{n-1})^{n-\ell}+(-1)^{n}=(-1)^{n}c_{n-1},
\end{array}
\right.
\]
where%
\[
\lambda_{1}\cdots\lambda_{n}=1\text{ and }c_{\ell}:=\frac{\rho_{\ell}}%
{\rho_{0}}=\frac{(\kappa_{\ell}-\lambda_{1})\cdots(\kappa_{\ell}-\lambda_{n}%
)}{(-1)^{n}\lambda_{1}\cdots\lambda_{n}}\neq0.
\]
From Lemma \ref{Kang2}, each $c=(c_{1},\cdots,c_{n-1})\in\mathcal{HP}%
_{III}=\operatorname{Im}(\Upsilon_{III})$ corresponds to only one
Henry-Parusi\'{n}ski invariant and thus induce a map from the analytic moduli
space to the Henry-Parusi\'{n}ski invariant, say $\overline{\Upsilon}%
_{III}:\mathcal{M}_{[n],2}^{0}\longrightarrow\mathcal{HP}_{III}$.

As before, let $c\in\mathcal{HP}_{III}$, then we say that $\Upsilon_{III}%
^{-1}(c)$ is a \textit{degenerate fiber} of $\Upsilon_{III}$ if $c_{j}=0$ for
some $j\in\{1,\cdots,n\}$ or else if the fiber has a multiply root; otherwise
we shall say that $\Upsilon_{III}^{-1}(c)$ is a \textit{non-degenerate fiber}.

\begin{theorem}
\label{type III}For type III functions in $\operatorname*{H}_{pq}^{d}$, the analytic moduli space with fixed
Henry-Parusi\'{n}ski invariant is determined by the images in $\frac
{\mathbb{V}^{n-1}}{\mathbb{Z}_{n}}$ of non-degenerate fibers of $\Upsilon
_{III}:\mathbb{V}^{n-1}\rightarrow\mathbb{C}^{n-1}$. More precisely, for each
$c\in\mathcal{HP}_{III}$
there exist $\#[\Upsilon_{III}^{-1}(c)]\leq n^{n-2}$ analytic types of function-germs
with the same Henry-Parusi\'{n}ski invariant $c$. The equality holds for
generic polynomials in $\mathcal{Z}\subset\operatorname*{H}_{pq}^{d}$.
\end{theorem}

\begin{proof}
From the description of the space $\mathcal{M}_{[n],2}^{0}$ (cf. (\ref{eq23}))
and also from (\ref{eq19}) and (\ref{eq18}), the correspondence between the
Henry-Parusi\'{n}ski invariant and the analytic invariants is determined by
the orbits of the action of $\mathbb{Z}_{n}$ on the fibers over $\mathcal{HP}%
_{III}$
of the map $\Upsilon_{III}:\mathbb{C}^{n-1}\longrightarrow\mathbb{C}^{n-1}$
$\ $given in coordinates by $\Upsilon_{III}(\kappa)=(\Upsilon_{1}%
(\kappa),\cdots,\Upsilon_{n-1}(\kappa))$, where%
\begin{equation}
\Upsilon_{j}(\kappa):=(\kappa_{j})^{n}+\sum_{\ell=1}^{n-1}(-1)^{\ell}\frac
{n}{n-\ell}\sigma_{\ell}(\kappa)(\kappa_{j})^{n-\ell}.\label{eq26}%
\end{equation}
\end{proof}
As a consequence of the above theorems, we are able to
generalize the main result in \cite{FeRu2011} as follows.

\begin{theorem}
\label{cont. family}Let $f_{t}$ be a continuous
family of germs of quasi-homogeneous (and not homogeneous) functions with
isolated singularity and constant Henry-Parusi\'{n}ski invariant then $f_{t}$
is analytically trivial.
\end{theorem}


\begin{proof}
First recall from \cite{FeRu2011} that it suffices to consider the commode
quasi-homogeneous case. Let $H_{p,q}^{d}$ denote the set of commode
quasi-homogeneous polynomials with weights $(p,q)$ and quasi-homogeneous
degree $d$. Let $\left[  H_{p,q}^{d}\right]  $ denote the set of analytic
conjugacy classes in $H_{p,q}^{d}$ and $\pi:H_{p,q}^{d}\longrightarrow\left[
H_{p,q}^{d}\right]  $ the natural projection, then we have the commutative
diagram%
\[%
\xymatrix{
& H_{p,q}^{d} \ar[d]^{\pi} \ar[rd]^{\Upsilon} \\
[0,1] \ar[ur]^{f_t}  \ar[r] & [H_{p,q}^{d}] \ar[r]^{\overline{\Upsilon}}
& \mathcal{HP}}%
\]
where $\Upsilon$ is the map in the statement of Theorems \ref{type II} and
\ref{type III}, and $\mathcal{HP}$ the space of Henry-Parusi\'{n}ski
invariants.
Suppose that $f_{t}$ is a continuous family with constant Henry-Parusi\'{n}ski
invariant $\rho$. Then $\pi\circ f_{t}$ is a continuous family of analytic
classes contained in $\left[  \overline{\Upsilon}\right]  ^{-1}(\rho)$. Since
$\left[  \overline{\Upsilon}\right]  $ is a ramified covering, then $\left[
\overline{\Upsilon}\right]  ^{-1}(\rho)$ is discrete and the result follows.
\end{proof}

\section{Examples}

Now let us study some examples.

\begin{example}
Suppose $n=2$, then $\kappa=\kappa_{1}$ and $\mu_{1}=\sigma_{1}(\lambda
)=\frac{2}{2-1}\sigma_{1}(\kappa)=2\kappa_{1}$. Thus%
\[
(\kappa_{1})^{2}+(-1)^{1}(2\kappa_{1})(\kappa_{1})^{2-1}=c_{1}%
-1\Longleftrightarrow(\kappa_{1})^{2}=1-c_{1}.
\]
Equivalently,%
\[
\kappa_{1}=\pm\sqrt{1-c_{1}.}%
\]
Since the square roots of unity are given by $\pm1$, then the equivalence
(\ref{eq23}) ensures that both points correspond to just one analytic class
$\left[  \lambda\right]  $.
\end{example}

\begin{example}
Suppose $n=3$, then $\kappa=(\kappa_{1},\kappa_{2})$ and $\mu_{j}=\sigma
_{j}(\lambda)$. Hence%
\[
\left\{
\begin{array}
[c]{c}%
(\kappa_{1})^{3}+(-1)^{1}\frac{3}{3-1}\sigma_{1}(\kappa)(\kappa_{1}%
)^{2}+(-1)^{2}\frac{3}{3-2}\sigma_{2}(\kappa)\kappa_{1}=c_{1}-1,\\
(\kappa_{2})^{3}+(-1)^{1}\frac{3}{3-1}\sigma_{1}(\kappa)(\kappa_{2}%
)^{2}+(-1)^{2}\frac{3}{3-2}\sigma_{2}(\kappa)\kappa_{2}=c_{2}-1,
\end{array}
\right.
\]
or equivalently%
\[
\left\{
\begin{array}
[c]{c}%
(\kappa_{1})^{3}-\frac{3}{2}(\kappa_{1}+\kappa_{2})(\kappa_{1})^{2}%
+3(\kappa_{1}\kappa_{2})\kappa_{1}=c_{1}-1,\\
(\kappa_{2})^{3}-\frac{3}{2}(\kappa_{1}+\kappa_{2})(\kappa_{2})^{2}%
+3(\kappa_{1}\kappa_{2})\kappa_{2}=c_{2}-1.
\end{array}
\right.
\]
In other words,%
\[
\left\{
\begin{array}
[c]{c}%
-\frac{1}{2}(\kappa_{1})^{3}+\frac{3}{2}\kappa_{2}(\kappa_{1})^{2}=c_{1}-1,\\
-\frac{1}{2}(\kappa_{2})^{3}+\frac{3}{2}\kappa_{1}(\kappa_{2})^{2}=c_{2}-1.
\end{array}
\right.
\]
Multiplying the first equation by $-2$ and summing it to the second one
multiplied by $2$, we have%
\[
(\kappa_{1}-\kappa_{2})^{3}=2(c_{2}-c_{1}).
\]
Let $\alpha$ be a cubic root of $2(c_{2}-c_{1})$ and $\omega=\exp
(2\pi\boldsymbol{i}/3)$, then the cubic roots of $2(c_{2}-c_{1})$ are of the
form $\omega\alpha,\omega^{2}\alpha,\omega^{3}\alpha=\alpha$. Hence
$\kappa_{1}-\kappa_{2}=\omega^{s}\alpha$, $s=0,1,2.$ Back to the system of
equations, we have%
\begin{align*}
2(1-c_{1})  &  =(\kappa_{1})^{3}-3(\kappa_{1})^{2}\kappa_{2}=(\kappa_{1}%
)^{2}(\kappa_{1}-3\kappa_{2})=(\kappa_{1})^{2}[3(\kappa_{1}-\kappa
_{2})-2\kappa_{1}]\\
&  =(\kappa_{1})^{2}[3\omega^{s}\alpha-2\kappa_{1}]=-2(\kappa_{1})^{3}%
+3\omega^{s}\alpha(\kappa_{1})^{2},
\end{align*}
i.e.,%
\begin{equation}
(\kappa_{1})^{3}-\frac{3\alpha\omega^{s}}{2}(\kappa_{1})^{2}+(1-c_{1})=0.
\label{eq27}%
\end{equation}
In order to be more specific, let us pick $c_{1}:=-1/3$ and $c_{2}:=1129/729$.
Then%
\[
\alpha=\sqrt[3]{2(c_{2}+1/3)}=\sqrt[3]{2\cdot\frac{1129+3^{5}}{3^{6}}%
}=\sqrt[3]{2\cdot\frac{1372}{3^{6}}=}\sqrt[3]{\frac{14^{3}}{9^{3}}=}\frac
{14}{9}.
\]
In this case, equations (\ref{eq27}) assume the form%
\[
(\kappa_{1})^{3}-\frac{7\omega^{s}}{3}(\kappa_{1})^{2}+\frac{4}{3}=0.
\]
Then for $s=0$ the possible values of $\kappa_{1}$ are given by the solutions
of the cubic equation $z^{3}-\frac{7}{3}z^{2}+\frac{4}{3}=0$, which are
$\{2,1,-2/3\}$. Since $\kappa=(\kappa_{1},\kappa_{2})=(\kappa_{1},\kappa
_{1}-\omega^{s}\alpha)$, then%
\[
(s=0):\kappa=\left\{
\begin{array}
[c]{l}%
(2,2-14/9)=(2,4/9);\\
(1,1-14/9)=(1,-5/9);\\
(-2/3,-2/3-14/9)=(-2/3,-20/9).
\end{array}
\right.
\]
For each $s=1,2,$ the solutions are of the form%
\[
(s=1):\kappa=\left\{
\begin{array}
[c]{l}%
\left(  2e^{\frac{2\pi\boldsymbol{i}}{3}},\frac{4}{9}e^{\frac{2\pi
\boldsymbol{i}}{3}}\right) \\
\left(  e^{\frac{2\pi\boldsymbol{i}}{3}},-\frac{5}{9}e^{\frac{2\pi
\boldsymbol{i}}{3}}\right) \\
\left(  -\frac{2}{3}e^{\frac{2\pi\boldsymbol{i}}{3}},-\frac{20}{9}%
e^{\frac{2\pi\boldsymbol{i}}{3}}\right)
\end{array}
\right.  (s=2):\kappa=\left\{
\begin{array}
[c]{l}%
\left(  2e^{\frac{4\pi\boldsymbol{i}}{3}},\frac{4}{9}e^{\frac{4\pi
\boldsymbol{i}}{3}}\right) \\
\left(  e^{\frac{4\pi\boldsymbol{i}}{3}},-\frac{5}{9}e^{\frac{4\pi
\boldsymbol{i}}{3}}\right) \\
\left(  -\frac{2}{3}e^{\frac{4\pi\boldsymbol{i}}{3}},-\frac{20}{9}%
e^{\frac{4\pi\boldsymbol{i}}{3}}\right)  .
\end{array}
\right.
\]

In order to be more precise, let us write down the explicit expressions of the
functions. First recall from (\ref{eq17}) that%
\[
f(x,y)=y^{3p}-\frac{3}{2}(\kappa_{1}+\kappa_{2})y^{2p}x^{q}+3\kappa_{1}%
\kappa_{2}y^{p}x^{2q}-x^{3q}.
\]
Thus for each value of $\kappa$ we have the following quasi-homogeneous
functions in three distinct analytic classes%
\[
\kappa=\left\{
\begin{array}
[c]{lll}%
(2,4/9) & \Longrightarrow & f^{1}(x,y)=y^{3p}-\frac{33}{9}y^{2p}x^{q}%
+\frac{24}{9}y^{p}x^{2q}-x^{3q},\\
(1,-5/9) & \Longrightarrow & f^{2}(x,y)=y^{3p}-\frac{6}{9}y^{2p}x^{q}%
-\frac{15}{9}y^{p}x^{2q}-x^{3q},\\
(-2/3,-20/9) & \Longrightarrow & f^{3}(x,y)=y^{3p}+\frac{39}{9}y^{2p}%
x^{q}+\frac{40}{9}y^{p}x^{2q}-x^{3q},
\end{array}
\right.
\]
whose polars are given by%
\[
\kappa=\left\{
\begin{array}
[c]{lll}%
(2,4/9) & \Longrightarrow & \partial_{y}f^{1}(x,y)=3py^{p-1}(y^{2p}-\frac
{22}{9}y^{p}x^{q}+\frac{8}{9}x^{2q}),\\
(1,-5/9) & \Longrightarrow & \partial_{y}f^{2}(x,y)=3py^{p-1}(y^{2p}-\frac
{4}{9}y^{p}x^{q}-\frac{5}{9}x^{2q}),\\
(-2/3,-20/9) & \Longrightarrow & \partial_{y}f^{3}(x,y)=3py^{p-1}(y^{2p}%
+\frac{26}{9}y^{p}x^{q}+\frac{40}{27}x^{2q}).
\end{array}
\right.
\]
From (\ref{eq17}) and (\ref{eq18}), we have $\rho_{0}=1$ and
\[
\rho_{j}=(\kappa_{j}-\lambda_{1})(\kappa_{j}-\lambda_{2})(\kappa_{j}%
-\lambda_{3})=\kappa_{j}^{3}-\sigma_{1}(\lambda)\kappa_{j}^{2}+\sigma
_{2}(\lambda)\kappa_{j}^{2}-1.
\]
Thus%
\[
\kappa=\left\{
\begin{array}
[c]{lll}%
(2,4/9) & \Longrightarrow & \left\{
\begin{array}
[c]{l}%
\rho_{1}=\kappa_{1}^{3}-\frac{33}{9}\kappa_{1}^{2}+\frac{24}{9}\kappa
_{1}-1=-21/9\\
\rho_{2}=\kappa_{2}^{3}-\frac{33}{9}\kappa_{2}^{2}+\frac{24}{9}\kappa
_{2}-1=-329/729
\end{array}
\right. \\
(1,-5/9) & \Longrightarrow & \left\{
\begin{array}
[c]{l}%
\rho_{1}=\kappa_{1}^{3}-\frac{6}{9}\kappa_{1}^{2}-\frac{15}{9}\kappa
_{1}-1=-21/9\\
\rho_{2}=\kappa_{2}^{3}-\frac{6}{9}\kappa_{2}^{2}-\frac{15}{9}\kappa
_{2}-1=-329/729
\end{array}
\right. \\
(-2/3,-20/9) & \Longrightarrow & \left\{
\begin{array}
[c]{l}%
\rho_{1}=\kappa_{1}^{3}+\frac{39}{9}\kappa_{1}^{2}+\frac{40}{9}\kappa
_{1}-1=-21/9\\
\rho_{2}=\kappa_{2}^{3}+\frac{39}{9}\kappa_{2}^{2}+\frac{40}{9}\kappa
_{2}-1=-329/729
\end{array}
\right.
\end{array}
\right.
\]
Therefore, $f^{1}$, $f^{2}$ and $f^{3}$ have the same Henry-Parusi\'{n}ski
invariant but represent three distinct analytic classes.
\end{example}

\end{document}